\documentclass[11pt,reqno]{amsart}

\usepackage{amssymb}
\usepackage{graphicx}
\usepackage{hyperref}

\theoremstyle{plain}
\newtheorem{theorem}{Theorem}[section]
\newtheorem{lemma}[theorem]{Lemma}

\theoremstyle{definition}

\newtheorem{example}{Example}

\numberwithin{equation}{section}
\numberwithin{theorem}{section}
\numberwithin{table}{section}
\numberwithin{figure}{section}

\newcommand{\T}{\mathcal T}
\newcommand{\pr}{\mathcal P}
\newcommand{\K}{\mathcal K}
\newcommand{\ang}[1]{\langle #1 \rangle}

\title{Solving Non-homogeneous Nested Recursions Using Trees}

\author{Abraham Isgur \and Mustazee Rahman \and Stephen Tanny}

\address{Department of Mathematics\\
University of Toronto\\
40 St. George Street\\
Toronto\\
ON M5S 2E4\\
Canada}

\email[Abraham Isgur]{umarovi@gmail.com}
\email[Mustazee Rahman]{mustazee.rahman@utoronto.ca}
\email[Steve Tanny]{tanny@math.utoronto.ca}

\date{\today}



\begin{document}

\begin{abstract}
The solutions to certain nested recursions, such as Conolly's $C(n) = C(n-C(n-1))+C(n-1-C(n-2))$, with initial conditions $C(1)=1, C(2)=2$, have a well-established combinatorial interpretation in terms of counting leaves in an infinite binary tree. This tree-based interpretation, which has a natural generalization
to a $k$-term nested recursion of this type, only applies to homogeneous recursions, and only solves each recursion for one set of initial conditions
determined by the tree. In this paper, we extend the tree-based interpretation to solve a non-homogeneous version of the $k$-term recursion that includes a constant term. To do so we introduce a tree-grafting methodology that inserts copies of a finite tree into the infinite $k$-ary tree associated with the solution of the corresponding homogeneous $k$-term recursion. This technique can also be used to solve the given non-homogeneous recursion with various sets of initial conditions.
\end{abstract}

\maketitle

\section{Introduction} \label{sec1}
In this paper all values for the variables and parameters are integers unless otherwise specified.
For $k \geq 1$, $a_i$, and $b_i > 0,\; i = 1 \ldots k$, consider the nested (also called meta-Fibonacci)
homogeneous recursion \begin{equation} \label{nested} A(n) = \sum_{i=1}^k A(n - a_i - A(n-b_i))\end{equation}
which we abbreviate as $\ang{a_1;b_1: \cdots: a_k;b_k}$. We call the sequences that appear as solutions to nested recursions \emph{meta-Fibonacci sequences}.

Over the past twenty years, many special cases of (\ref{nested}), together with alternative sets of initial conditions, have been examined (see the references for specifics). Examples include Hofstadter's famous and mysterious $Q$-sequence \cite{GEB} given by $\ang{0;1:0;2}$ with $Q(1) = Q(2) = 1$, and Conolly's well-known sequence $C(n)$ \cite{Con} given by $\ang{0;1:1;2}$ with $C(1) =1, C(2) = 2$. Recently, fascinating and unexpected combinatorial connections have been discovered between the solutions to certain such nested recursions and infinite, labeled trees \cite{BLT, DR, JR, IRT}. For example, it is shown in \cite{JR} that the shifted Conolly sequence $C_s(n)$ determined by $\ang{s;1:s+1;2}$ for any fixed $s \geq 0$ and initial conditions $C_s(i) = 1$ for $1 \leq i \leq s+1$ and $C_s(s+2) = 2$ counts the number of leaves in a suitably constructed infinite labeled binary tree (with root at infinity) that have labels that are less than or equal to $n$ (the construction depends on the parameter $s$). In the labeling, each node of the infinite binary tree receives one label except for the so-called $s$-nodes along the top of the tree, each of which receives $s$ labels. In \cite{DR} an analogous combinatorial interpretation is derived for solutions to the $k$-term recursion (\ref{nested}) with $a_i = s+i-1$ and $b_i= i$, and with $k+s$ initial conditions that are determined by the leaf counts of the correspondingly constructed infinite, labeled $k$-ary tree. These initial conditions are said to ``follow the tree", in the sense that they are precisely the ones that force the solution to conform to the specified labeled tree.

Building on this work, Isgur et al \cite{IRT} vary the labeling scheme by inserting $j$ labels in each node of the $k$-ary tree rather than a single label, where $j \geq 1$ is a fixed parameter. In this way they derive a combinatorial interpretation for the solution to the generalized Conolly recurrence \begin{equation} \label{Conolly} R(n) = \sum_{i=1}^kR(n-s-(i-1)j-R(n-ij)) \end{equation} with the initial conditions: $R(i) = i$ for $1 \leq i \leq j$, $R(i) = j$ for $j+1 \leq i \leq j+s$, $R(i) = i-s$ for $j+s+1 \leq i \leq kj + s$, and $R(i) = kj$ for $kj+s+1 \leq i \leq (k+1)j+ 2s$. It is shown that $R(n)$ counts the number of \emph{labels} in the leaves of the labeled $k$-ary tree that are less than or equal to $n$.

Here we extend the tree-based correspondence described above to combinatorially interpret solutions to a non-homogeneous version of the Conolly nested recursion (\ref{Conolly}), namely, \begin{equation} \label{recursion} R(n) = \sum_{i=1}^kR(n-s-(i-1)j-R(n-ij))+ \nu \end{equation} where $\nu$ is any constant, and with specified initial conditions. Interest in such nested recursions is natural and longstanding; see, for example, \cite{Gol}, where the recursion $g(n) = g(n-g(n-1)) + 1$, with $g(1) = 1$, is shown to have a neat, closed form solution. Our focus on a constant for the non-homogeneous term in (\ref{recursion}) can be readily explained: if the non-homogeneous term is an integer valued function $\nu(n)$ with $|\nu(n)| \geq cn$ for all $n$, then the right side of (\ref{recursion}) grows at least linearly in $n$. Therefore $R(n)$ will grow at least linearly in $n$ and it seems plausible that eventually one of the arguments $n-s-(i-1)j-R(n-ij)$ in (\ref{recursion}) will be negative or will exceed $n$. At that point the recursion will cease to be well-defined. So a constant value for $\nu(n)$ is a natural choice. \footnote{Of course, one could consider non-constant, sub-linear $\nu(n)$. To date our empirical investigations suggest that only constant $\nu(n)$ lead to well-defined, infinite solution sequences. Considerable further work is needed in this area to confirm this contention, or to determine which non-constant $\nu(n)$, if any, lead to an infinite solution sequence to (\ref{recursion}) for some set of initial conditions.}

To solve (\ref{recursion}) combinatorially we ``graft" infinitely many copies of a finite, rooted tree $\T$ (or in some cases a portion of $\T$) onto the original $k$-ary tree that solves the related homogeneous recursion (\ref{Conolly}).
As we explain below, it turns out that for any given value of $\nu$ in (\ref{recursion}), we can find infinitely many finite trees $\T$ that correspond to that choice of $\nu$. Each of these finite trees determines a set of initial conditions for the recursion, and these sets of initial conditions may differ. Thus, our tree grafting technique permits us to find combinatorial interpretations for different sets of initial conditions for the same recursion. In particular, we can use our technique for the case $\nu =0$, thereby solving the homogeneous nested recursion (\ref{Conolly}) with initial conditions determined by the choice of $\T$. Prior to this, the only combinatorial solution to (\ref{Conolly}) is the one associated with the initial conditions imposed by the leaf counts of the usual $k$-ary tree \cite{DR, IRT}.

The outline of the rest of this paper is as follows. In Section \ref{sec2} we describe precisely the procedure for grafting copies of an arbitrary finite, rooted tree $\T$ on the infinite $k$-ary tree in \cite{DR, IRT}, and for labeling the resulting infinite tree $\K$. This construction depends upon $\T$, as well as the three parameters $k, j$ and $s$ in (\ref{recursion}). In Section \ref{sec3} we establish that the infinite tree $\K$ constructed in Section \ref{sec2} provides a combinatorial interpretation to (\ref{recursion}): $R(n)$ is the number of labels up to $n$ on leaves of $\K$. Finally, in Section \ref{sec4} we discuss alternative labeling schemes for $\K$ that give rise to a variety of interesting results; in particular, we derive a new combinatorial interpretation for Golomb's recursive sequence $g(n)$.

\section{Constructing the infinite tree $\K$: the grafting technique} \label{sec2}
Let $\T$ be any finite rooted tree with at least two nodes. The \textbf{height} $p$ of $\T$ is the length of the longest path from the root to any of its nodes. Fix $k \geq 1$ corresponding to the desired value of $k$ in ($\ref{recursion}$). We create a modified labeled $k$-ary tree $\K$ using $\T$.

The construction of $\K$ requires two steps. First we construct the nodes and edges, that is, the skeleton, of $\K$; this involves grafting copies of $\T$ on the infinite $k$-ary tree in \cite{DR, IRT}. Then we insert labels, which are successive positive integers, within the nodes of $\K$. To do this, first we specify the order in which the nodes of $\K$ are to be traversed one at a time; then we insert the appropriate number of labels, either $j$ or $s$ (the parameters in ($\ref{recursion})$) in each node. As we traverse $\K$, we keep count of the number of labels up to that point that are located in the leaves of $\K$. The ``leaf label" sequence generated by this enumeration satisfies a nested, non-homogeneous Conolly-type recursion of the form (\ref{recursion}), where the tree $\T$ determines $\nu$.

To help describe this construction, we illustrate our discussion using the rooted tree $\T$ of height 3 in Figure \ref{fig1}, together with $k=2$. We show how this results in an infinite, labeled binary tree with leaf label counting function $R(n)$ that satisfies the recursion $R(n) = R(n-R(n-2))+R(n-2-R(n-4))-2$.

\begin{figure}[htpb]
\begin{center}
\includegraphics[scale=0.35]{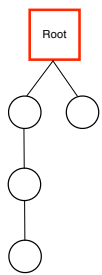}
\caption{Example of a rooted tree $\T$ of height $p=3$.} \label{fig1}
\end{center}
\end{figure}

\subsection*{Constructing the skeleton of $\K$}
The skeleton of $\K$ consists of an infinite sequence $\K_i$ of rooted, finite subtrees of $\K$ that we join together to form $\K$. For each $i$, the root of $\K_i$ is called the $i^{th}$ \textbf{supernode} of $\K$, while all other nodes are \textbf{regular nodes}. For $i=p>1$, where $p$ is the height of $\T$, the subtree $\K_p$ is isomorphic to $\T$. The subtree $\K_{p-1}$ is obtained by making a copy of $\K_p$ and then deleting all the leaves of the copy. See Figure \ref{fig3}, where we illustrate the subtrees $\K_3$ and $\K_2$ using the tree $\T$ of height $p=3$ in Figure \ref{fig1}; we draw the supernodes as squares and the regular nodes as circles. For $2<i\leq p$, we repeat this process successively, making $\K_{i-1}$ by copying $\K_i$ and deleting the leaves of the copy. The subtree $\K_1$ is a special case: after deleting the leaves of a copy of $\K_2$, we attach an extra regular node as a child of the first supernode (this extra child can be considered as the zeroth supernode - see Section \ref{sec4}).

For $i > p$ we construct the subtrees $\K_i$  by first making a copy of $\K_{i-1}$ and then attaching precisely $k$ nodes to each of its leaves. Thus, for $i > p$ the subtree $\K_i$ consists of the tree $\T$ with each of its leaves the root of a $k$-ary subtree of height $i-p$, so $\K_i$ has height $i$. Finally, for all $i>0$ we connect all the subtrees $\K_i$ to $\K_{i+1}$ by adding an edge from the $i^{th}$ supernode to the $(i+1)^{st}$. See Figure \ref{fig3}. Notice that the tree $\K$ is closely related to the infinite $k$-ary tree in \cite{DR, IRT}: take the former $k$-ary tree, insert the finite tree $\T$ (or a portion of it), with the root of $\T$ coinciding with each of the supernodes of the $k$-ary tree, and leave the rest of the original $k$-ary tree untouched.

\begin{figure}[htpb]
\begin{center}
\includegraphics[scale=0.35]{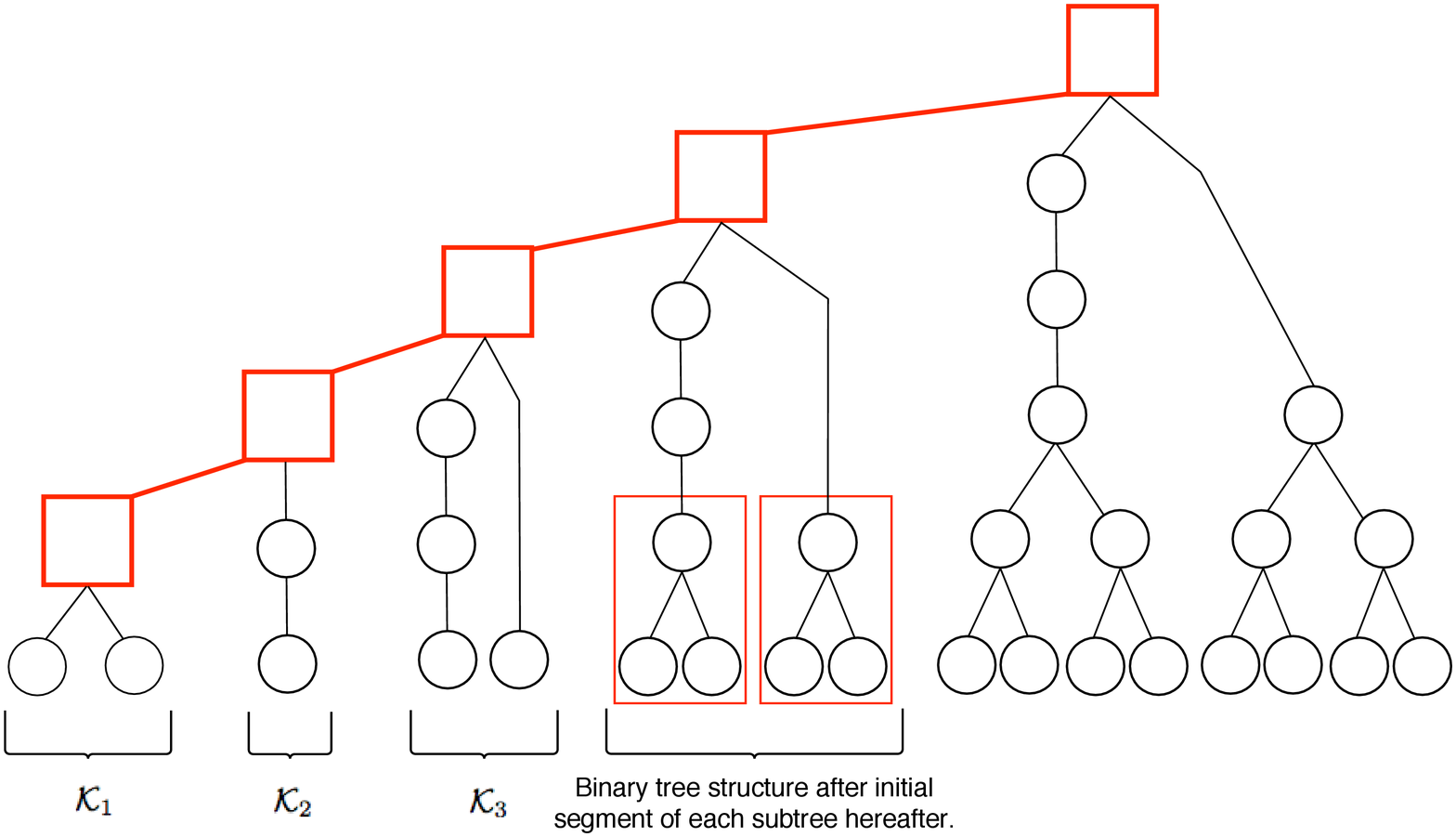}
\caption{First 5 subtrees of the skeleton of $\K$ derived from the tree in Figure \ref{fig1}, using $k=2$.} \label{fig3}
\end{center}
\end{figure}

\subsection*{Labeling $\K$}
Fix the values $j \geq 1$ and $s \geq 0$, corresponding to the parameters of the same name in ($\ref{recursion}$). Insert $j$ labels into each regular node of $\K$ and $s$ labels into each supernode; for convenience, we refer to the regular nodes and supernodes as $j$-nodes and $s$-nodes, respectively. The labels consist of successive integers starting at 1. Before we can insert these labels we must specify the traversal order of the nodes in $\K$. We recursively define a \textbf{pre-order} traversal as follows: $\K_1$ is traversed by beginning at the first child of the first supernode followed by the supernode itself and then its remaining children (note this is not the pre-order traversal of $\K_1$). Having traversed $\K_i$ for $i \geq 1$, the subtree $\K_{i+1}$ is traversed next in the usual pre-order way by beginning at its root, which is the $(i+1)^{th}$ supernode. See Figure \ref{fig4}, where we label the nodes of $\K$ in the order in which they are traversed, and then Figure \ref{fig5}, where we insert $j=2$ labels in each regular node and $s=0$ labels in each supernode.

\begin{figure}[htpb]
\begin{center}
\includegraphics[scale=0.35]{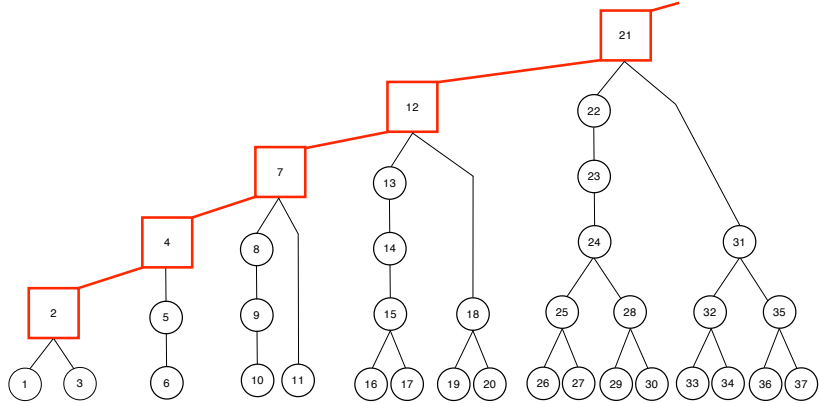}
\caption{The order in which the nodes of $\K$ from Figure \ref{fig3} are traversed.} \label{fig4}
\end{center}
\end{figure}

\begin{figure}[htpb]
\begin{center}
\includegraphics[scale=0.35]{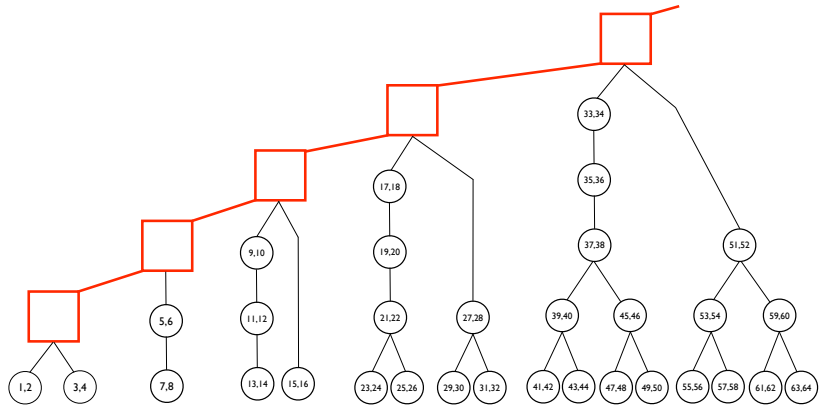}
\caption{First 5 subtrees of the completed infinite tree $\K$ from our example. The labeling parameters are $j=2$ and $s=0$.} \label{fig5}
\end{center}
\end{figure}

We explain in the next section how the combinatorial interpretation for the solution of (\ref{recursion}) is derived from the ``leaf label" sequence generated from $\K$. Before doing so, we require additional terminology and notation with which we conclude this section.

Call $\K(n)$ the subtree of $\K$ containing all the labels between $1$ and $n$ and all the nodes in pre-order up to the node containing $n$. For
$m \geq 1$, define $m$ to be a \textbf{leaf label} of $\K$ if $m$ is contained in a leaf of $\K$. Throughout the paper the \textbf{leaf labels of $\K(n)$} are defined to be all the labels in $\K(n)$ that are leaf labels in $\K$. It is very important to note that a node may be a leaf in $\K(n)$ and not be a leaf
in $\K$. For example, in Figure \ref{fig5}, the node containing the label 21 is a leaf of $\K(21)$ but not of $\K$. Thus, 21 is \emph{not} a leaf label.

Let $R(n)$ be the number of leaf labels in $\K(n)$. In Figure \ref{fig5}, $R(7) =5$ and $R(20) = 10$.

A \textbf{penultimate node} of $\K$ is a non-leaf node in $\K$ such that all of its children are leaves (for example, nodes 2 and 18 in Figure \ref{fig4} but not node 7). Call the labels in a penultimate node \textbf{penultimate labels}. The penultimate nodes (respectively, labels) of $\K(n)$ are the penultimate
nodes (respectively, labels) of $\K$ that are included in $\K(n)$.

Note that for $i \geq 2$ the leaves of $\K_i$ are the penultimate nodes of $\K_{i+1}$; the leaves of $\K_1$, other than the first leaf, are the penultimate nodes of $\K_2$, and $\K_1$ always has exactly one penultimate node (the first $s$-node of $\K$). Let $\ell_i$ be the number of leaves in $\K_i$. Define $\alpha$ (respectively, $\beta$) as the number of leaf labels (respectively, penultimate labels) occurring in $\K_1$ through $\K_p$. Then by the preceding observation $\alpha = j\sum_{i=1}^p \ell_i$, and $\beta = j(\sum_{i=1}^{p-1} \ell_i -1) + s$. Finally, let $N(i)$ be the largest label of $\K$ that occurs in $\K_i$.

We are now prepared to state and prove our key finding.

\section{Solving the non-homogeneous Conolly nested recursion} \label{sec3}

\begin{theorem} \label{thm1}
Let $\T$ be a finite rooted tree of height $p$. Let $\K$ be the infinite tree constructed using $\T$ and fixed parameters $k \geq 1, j \geq 1$, and $s \geq 0$. Define $\nu = \alpha - k(\beta - s + j)$. Let $R(n)$ be the leaf label counting function of $\K$. Then for $n > N(p+1)$, $R(n)$ satisfies the non-homogeneous nested recursion (\ref{recursion}), that is, \[R(n) = \displaystyle \sum_{i=1}^k R(n-s-(i-1)j-R(n-ij)) + \nu\,.\]
Equivalently, if any function $L(n)$ is defined by (\ref{recursion}) and the first $N(p+1)$ values of $L(n)$ agree with the corresponding values for the leaf label counting function $R(n)$, then $L(n)=R(n)$ for all $n$.
\end{theorem}

In Figure \ref{fig5} $p=3, k=j=2, s=0$ and $\nu = 10 - 2( 4 - 0 + 2) = -2$. Then for $n>N(4)=32$, the leaf label counting function $R(n)$ satisfies $R(n) = R(n-R(n-2)) + R(n-2-R(n-4)) - 2$.

Before turning to the proof of Theorem \ref{thm1}, we provide several observations. From the formulas for $\alpha$ and $\beta$ in Section \ref{sec2} we have a computationally simpler expression for $\nu$:
\begin{equation*} \nu = j\ell_p - j(k-1)(\ell_1 + \cdots + \ell_{p-1})\;\text{if}\; p \geq 2,\;\text{and}\; \nu = j(\ell_1-k)\;\text{if}\; p =1\,.\end{equation*}

In some cases fewer than $N(p+1)$ initial conditions will suffice. For our purposes, we are only interested in knowing that for some sufficiently large number of initial conditions the recursion (\ref{recursion}) will generate the leaf label counting sequence as its solution.
Finally, note that different choices of $\T$ enable us to solve recursions of the form (\ref{recursion}) with diverse initial conditions. In particular, now we are able to solve (\ref{Conolly}) with many different sets of initial conditions.

We begin the proof of Theorem \ref{thm1} by defining the \emph{pruning operation} on the subtree $\K(n)$ for $n > N(1)$. This operation yields a new tree $\pr\K(n)$ defined as follows: first, delete all leaf labels of $\K(n)$ along with the nodes containing them. Then convert the first $s$-node into a $j$-node. Finally, relabel the new tree in pre-order, keeping in mind that the first $s$-node is now a $j$-node so it receives $j$ labels rather than $s$. See Figure \ref{fig6} for the pruning of $\K(27)$ from the tree in Figure \ref{fig5}. Note that the node of $\K(27)$ that contains the label 27 is a leaf of $\K(27)$ but not of $\K$, and as such it is not deleted.

The significance of the pruning operation on the subtree $\K(n)$ is that it results in $\K(m)$ for some $m < n$. In this regard, we can view $\K$ as self-similar with respect to the pruning operation. Let $\pr R(n)$ denote the number of leaf labels in $\pr\K(n)$. We build to the proof of Theorem \ref{thm1} via a series of lemmas concerning $\pr\K(n)$.

\begin{figure}[htpb]
\begin{center}
\includegraphics[scale=0.35]{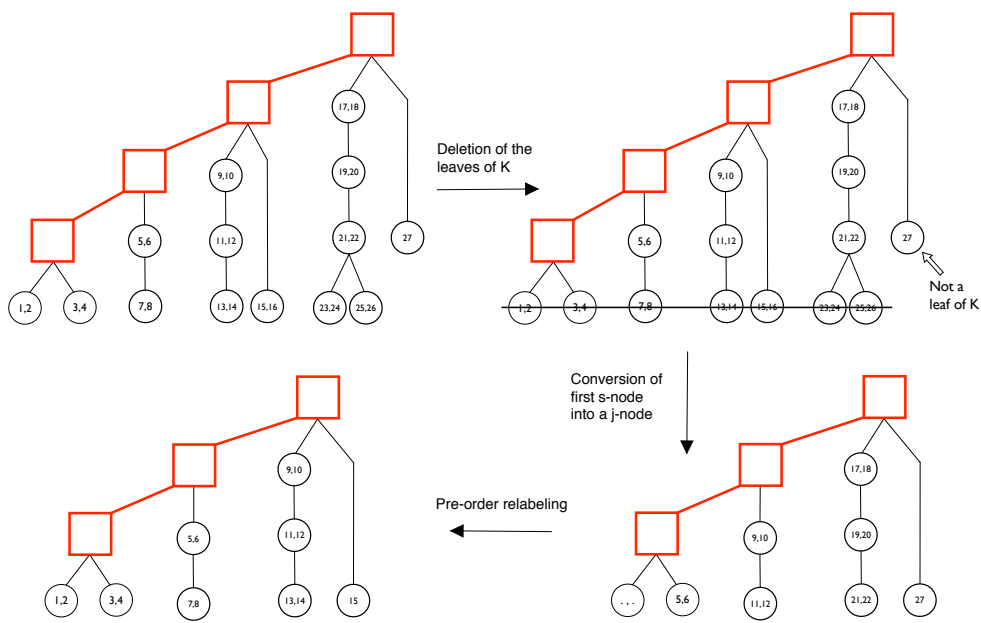}
\caption{The pruning operation on $\K(27)$ results in $\K(15)$.} \label{fig6}
\end{center}
\end{figure}

\begin{lemma}[Pruning] \label{lem1}
For $n > N(1)$ the tree $\pr K(n)$ has $n-s+j-R(n)$ labels and is isomorphic to the subtree $\K(n-s+j-R(n))$. Consequently, $\pr R(n) = R(n-s+j-R(n))$.\end{lemma}

\begin{proof} Since $\K(n)$ contains $R(n)$ leaf labels, deleting the leaf labels of $\K(n)$ results in a loss of $R(n)$ labels. Also, replacing the first $s$-node with a $j$-node results in a net change of $j-s$ labels following the pruning operation. Thus, the total number of labels in $\pr \K(n)$ is $n-R(n)-s+j$.

That $\pr K(n)$ is isomorphic to the subtree $\K(n-s+j-R(n))$ follows directly from the definition of the pruning operation and the construction of the tree $\K$, since deleting all the leaves of $\K_q$ results in $\K_{q-1}$. More generally, if we delete the leaves of $\K$ from the subtree of $\K_q$ that consists of the first $m$ nodes of $\K_q$ (in pre-order), and then relabel in pre-order, the result is the subtree consisting of the first $m' $ nodes of $\K_{q-1}$ for some $m' < m$.

Finally, since $\pr K(n)$ is isomorphic to $\K(n-s+j-R(n))$ and $\K(n-s+j-R(n))$ contains $R(n-s+j-R(n))$ leaf labels by definition, $\pr R(n) = R(n-s+j-R(n))$.\end{proof}

The key to the proof of Theorem \ref{thm1} is that since most penultimate nodes have $k$ children, $k$ times the number of penultimate labels in $\K(n)$ is essentially the number of leaf labels in $\K(n)$, with the difference being given by the non-homogeneous term $\nu$. Call $\K(n)$ \emph{complete} if each of its penultimate nodes has all of its children from $\K$, and each of these children has $j$ labels. If $\K(n)$ is complete then the number leaf labels in $\pr \K(n)$ is $j$ times the the number of penultimate level nodes in $\K(n)$.

\begin{lemma}[Completeness] \label{lem2}
For $n \geq N(p)$, if $\K(n)$ is complete then \[\pr R(n) = \dfrac{R(n) - \nu}{k}\,.\]\end{lemma}

\begin{proof} Recall that $\alpha$ (respectively, $\beta$) is the number of leaf labels (respectively, penultimate labels) occurring in $\K_1$ through $\K_p$. Since $n \geq N(p)$, the pruned tree $\pr K(n)$ contains the subtrees $\K_1$ to $\K_{p-1}$, so has $\beta -s+j$ leaf labels in these subtrees. So $\pr R(n) - (\beta - s+j)$ is the number of leaf labels in $\pr K(n)$ occurring after $\K_{p-1}$. But by the self-similarity of $\K$ with respect to pruning, this is also the number of penultimate labels in $\K(n)$ after label $N(p)$, so after $\K_p$.

By the completeness of $\K(n)$, each penultimate node of $\K(n)$ occurring after $\K_p$ has $k$ children, these children are the only leaves of $\K$ included in $\K(n)$ that occur after $\K_p$, and all these children, as well as their penultimate node parents, have $j$ labels each. Group these children with their penultimate level parents (who also come after $\K_p$). We then get a $k:1$ correspondence between these children and their parents, which also extends to a correspondence between the labels situated in them.

Now, $R(n) - \alpha$ counts the number of leaf labels in $\K(n)$ after label $N(p)$. So it equals the number of labels in the children mentioned in the preceding paragraph. On the other hand $\pr R(n) - (\beta - s+j)$, the number of penultimate labels in $\K(n)$ after label $N(p)$, is the number of labels in the parents mentioned above. Thus, we use the correspondence above to count the number of leaf labels in $\K(n)$ after label $N(p)$ in two ways:
\[ k(\pr R(n) - (\beta -s+j)) = R(n) - \alpha\,.\] Simplifying and substituting $\nu = \alpha - k(\beta-s+j)$ we get the desired result. Note that this also explains the definition of $\nu$.
\end{proof}

Now observe that if $\K(n)$ and $\K(m)$ have the same penultimate labels, then $\pr R(n) = \pr R(m)$. We use this to compute the values of $\pr R(n)$ based on the location of $n$ in $\K$. To this end, let $\Delta(n)$ denote the minimal non-negative integer such that $\K(n + \Delta(n))$ is complete.

\begin{lemma} \label{lem3}
The following holds for $n > N(p+1)$.\begin{enumerate}
\item If $n$ is neither a leaf label nor a penultimate label then $\K(n)$ is complete. Consequently for every $n > N(p+1)$, we have that $0 \leq \Delta(n) < (k+1)j$.\\
\item Suppose $\Delta(n) > 0$ so that $\K(n)$ is not complete. Then $0 < \Delta(n) < kj$ if and only if $n$ is a leaf label, and $\Delta(n) \geq kj$ if and only if $n$ is a penultimate label.\\
\item If $0 \leq \Delta(n) \leq kj$ then $\pr R(n) = \dfrac{R(n) + \Delta(n) - \nu}{k}$.\\
\item If $\Delta(n) > kj$ then $\pr R(n) = \dfrac{R(n) + kj - \nu}{k} - \Delta(n) + kj$.
\end{enumerate}
\end{lemma}

\begin{proof} (1) and (2): The first statement in (1) follows from the definition of completeness and the construction of $\K$. To prove the second part of (1) and assertion (2), note that if $\Delta(n) > 0$ then $n$ is a label on either a penultimate node or on one of the $k$ children of a penultimate node. In either case we can complete $\K(n)$ by adding any missing labels on the penultimate node, and nodes and labels for any missing children until the last label in the last child of said penultimate node. In either case we add up to (but excluding) $(k+1)j$ labels. This proves the second statement in (1). Further, the label $n$ is on a penultimate node if and only if $kj \leq \Delta(n) < (k+1)j$ (since in this case we must add the nodes and labels for all $k$ children); otherwise, $n$ is a label in some child on the bottom level, and $0 \leq \Delta(n) \leq kj$. This establishes (2).

(3): If $\Delta(n) = 0$ then this assertion is simply Completeness Lemma \ref{lem2}. If $0 < \Delta(n) \leq kj$ then from (2) we get that there exists a penultimate node $X$ such that either $n$ is a label in one of its $k$ children (when $0 < \Delta(n) < kj)$ or $n$ is the final label (in pre-order) of $X$ (when $\Delta(n) = kj$). In both cases all of the trees $\K(n), \ldots, K(n + \Delta(n))$ have the same penultimate labels, namely, all the penultimate labels from $1$ through to the final label in $X$. It follows, as we observed just prior to the statement of this lemma, that $\pr R(n) = \pr R(n+\Delta(n))$. Further, $R(n + \Delta(n)) = R(n) + \Delta(n)$, since the $\Delta(n)$ labels following $n$ are all leaf labels in $\K$. Since $\K(n + \Delta(n))$ is complete, we apply the Completeness Lemma \ref{lem2} to it to deduce that \begin{equation*}
\pr R(n) = \pr R(n + \Delta(n)) = \frac{R(n + \Delta(n)) - \nu}{k} = \frac{R(n) + \Delta(n) - \nu}{k}\,.\end{equation*}

(4): If $\Delta(n) > kj$ then using the same notation as in the previous paragraph we see from (2) that $n$ is a label of the penultimate node $X$ but it is not the last label of $X$. Let $n'$ be the last label of $X$ so $n'-n = \Delta(n) - kj$. Also $\pr R(n') = \pr R(n) + (n'-n) = \pr R(n) + \Delta(n) - kj$, and clearly $R(n) = R(n')$ and $\Delta(n') = kj$. It follows by (3), applied to $n'$, that \begin{eqnarray*}
\pr R(n) &=& \pr R(n') - \Delta(n) + kj \\
&=& \frac{R(n') + \Delta(n') - \nu}{k} - \Delta(n) + kj\\
&=& \frac{R(n') + kj - \nu}{k} - \Delta(n) + kj \;.\end{eqnarray*}
This proves (4) and completes the proof of the lemma.\end{proof}

To prove Theorem \ref{thm1} we demonstrate the following key relation: for $n > N(p+1)$,
\begin{equation} \label{keyrelation} R(n) - \nu = \displaystyle \sum_{i=1}^k \pr R(n-ij)\,.\end{equation}
From (\ref{keyrelation}) and the Pruning Lemma \ref{lem1} our desired result is immediate, since for $n > N(p+1)$,
\begin{eqnarray*} R(n) &=& \sum_{i=1}^k \pr R(n-ij) + \nu \\
&=& \sum_{i=1}^k R(n- s - (i-1)j - R(n-ij)) + \nu \,.\end{eqnarray*}

To prove relation (\ref{keyrelation}) we have two cases.

\paragraph{\textbf{Case 1:}} Suppose $n$ is a leaf label. Then there exists $q$ and $r$ such that $1 \leq q \leq k$ and $1 \leq r \leq j$ and $n$ is the $r^{th}$ smallest label on the $q^{th}$ child (in pre-order) of its parent node $X$ (a penultimate node). The trees $\K(n), \K(n-j), \ldots, \K(n-(q-1)j)$ all have the same penultimate labels consisting of all such labels up to and including the penultimate labels in $X$. The tree $\K(n-qj)$ ends on the $r^{th}$ label in $X$, so its penultimate labels differ from those of the previously mentioned trees only at the last $j-r$ labels of $X$. The trees $\K(n-(q+1)j), \ldots, \K(n- kj)$ do not end on penultimate nodes, so they all have the same penultimate labels, namely, all penultimate labels occurring before the labels in $X$.

We now apply the remark we made just prior to Lemma \ref{lem3} that if two trees have the same penultimate labels, their pruned trees have the same number of leaf labels. So $\pr R(n) = \pr R(n-ij)$ for $1 \leq i \leq q-1$. In the same way, $\pr R(n-qj) = \pr R(n) - (j-r)$ and $\pr R(n-ij) = \pr R(n) - j$ for $q+1 \leq i \leq k$. But $\Delta(n) = (j-r) + j(k-q)$, so by (3) of Lemma \ref{lem3} we get $\pr R(n) = \frac{R(n) + (j-r) + j(k-q) - \nu}{k}$. Thus we conclude that
\[ \sum_{i=1}^k \pr R(n-ij) = k \pr R(n) -(j-r)-j(k-q) = R(n) - \nu\,.\]

\paragraph{\textbf{Case 2:}} Suppose $n$ is not a leaf label. In this case the subtrees $\K(n-j), \dots, \K(n-kj)$ all have the same penultimate labels, so $\pr R(n-j) = \pr R(n-ij)$ for $1 \leq i \leq k$. The subtrees $\K(n)$ and $\K(n-j)$ may differ on at most one penultimate node, which happens precisely when $n$ lies on a penultimate node $X$. So we can write $\pr R(n-j) = \pr R(n) - r'$ where $0 \leq r' \leq j$. Here $r'=0$ if and only if $K(n)$ is complete, and otherwise $n$ is the $r'$-th smallest label on the penultimate node $X$.

If $r' = 0$ then $\pr R(n) = \frac{R(n) - \nu}{k}$ by the Completeness Lemma \ref{lem2}. If $r' > 0$ then $\Delta(n) = kj + (j-r)$, and (4) of Lemma \ref{lem3} implies that $\pr R(n) = \frac{R(n) - \nu}{k} + r'$ after simplification. Therefore in all cases we conclude that
\begin{eqnarray*} \sum_{i=1}^k \pr R(n-ij) = k \pr R(n-j) &=& k(\pr R(n) - r')\\
&=& k \left ( \dfrac{R(n) - \nu}{k} + r' - r' \right ) \\
&=& R(n) - \nu\,.\end{eqnarray*}
This completes the proof of Case 2, and the theorem.

\section{Further applications} \label{sec4}

In the construction of the tree $\K$, we created each subtree $\K_i$ by starting with a complete $k$-ary tree of height $i$ and inserting an arbitrary tree $\T$. Here we describe how slight modifications to this construction, such as to the labeling scheme or to the number of labels in various nodes, can still yield a tree $\K$ whose leaf label counting function satisfies a recursion with the form (\ref{recursion}). The key requirement to these modifications is that they preserve the self-similarity of $\K$ with respect to a suitably adapted pruning operation (or in other words, provided that removing the leaves of $\K$ results in a tree isomorphic to $\K$ up to some consistent finite correction).

In what follows, instead of stating a complicated theorem describing the most general possible modification that we can devise, we illustrate the flexibility of our methodology and its ability to produce interesting results via several examples.

\begin{example}[Solving (\ref{recursion}) with arbitrary values of $s$] \label{ex1} We begin by describing how to adjust the labeling of $\K$ to yield a combinatorial interpretation for solutions to (\ref{recursion}) with arbitrary values for the parameter $s$. Our approach turns out to be somewhat simpler than that of \cite{DR}, where this is accomplished for (\ref{Conolly}), the homogeneous version of (\ref{recursion}), by \emph{removing} labels from specific nodes in the tree when $s<0$.

We change the number of labels inserted within each supernode of $\K$: let the $m^{th}$ supernode receive $s_m \geq 0$ labels. Next, let the extra child of the first supernode receive $s_0 \geq 0$ labels (instead of $j$). We now derive the resulting recursion related to $\K$. To do so, we must identify the nature of the pruning operation associated with $\K$.

For any label $n$, suppose that $n$ lies in the subtree $\K_m$ of $\K$, where $m= m(n)$. Prune the subtree $\K(n)$ as follows: delete all the leaf labels and the nodes containing them. Replace the $s_i$ labels in the $i^{th}$ supernode by $s_{i-1}$ labels for each $1 \leq i \leq m$. Then relabel the new tree $\pr K(n)$ in the usual way by pre-order. The tree $\pr \K(n)$ contains $n - R(n) + (s_0-s_1) + \cdots + (s_m-s_{m-1}) = n - R(n) +s_0 - s_m$ labels, and it is isomorphic to the subtree $K(n-R(n)+s_0 - s_m)$. Analogous to the Pruning Lemma \ref{lem1}, we have that $\pr R(n) = R(n-R(n) + s_0 - s_m)$.

Similarly, we have the analogue of the Completeness Lemma \ref{lem2} with the new value of $\nu = \alpha - k(\beta +s_0 - s_1)$ (where $\alpha$ and $\beta$ retain the same meaning as before). In the same way, Lemma \ref{lem3} and the key relation (\ref{keyrelation}) continue to hold as before. Thus, we conclude that the leaf label counting function $R(n)$ satisfies
\begin{equation} \label{non-homo_modif1} R(n) = \sum_{i=1}^k R(n-(s_{m(n-ij)}-s_0) - ij - R(n-ij)) + \nu \quad \text{for}\; n > N(p+1)\,.\end{equation} Notice that as $i$ ranges from $1$ to $k$, $m(n-ij)$ can only take the values $m(n)-1$ or $m(n)$, since jumping back by $ij$ labels for $1 \leq i \leq k$ takes us at worst to the previous subtree $\K_{m(n)-1}$.

When $s_0 = t+j$ and $s_m = s$ for all $m \geq 1$ then $s_m(n-ij) = s$ for all $n > N(p+1)$, and we deduce after some simplification that
\begin{equation} \label{non-homo_modif2} R(n) = \sum_{i=1}^k R(n-(s-t) - (i-1)j - R(n-ij)) + \nu \end{equation} for $n > N(p+1)$. The parameter $s-t$ can take any integer value, whereas the equivalent parameter $s$ in (\ref{recursion}) had to be nonnegative.
\end{example}

For the next application of our methodology, we apply the idea in Example \ref{ex1}, together with a modified labeling scheme, to solve (\ref{Conolly}) with specified initial conditions. We illustrate our approach with $k=2$, so with the recursion \begin{equation} \label{2-ary} R(n) = R(n-s-R(n-j)) + R(n-s-j-R(n-2j))\,.\end{equation} As we discussed in Section \ref{sec1}, this recursion, together with initial conditions that follow the corresponding labeled binary tree, is solved in \cite{IRT}.

\begin{example} [Solving (\ref{2-ary}) with more general initial conditions] \label{ex2}
We demonstrate how to solve (\ref{2-ary}) with initial conditions that begin with a string of $s_1 +1$ 1s for any given $s_1 \geq 0$. These are followed by an additional $s + 5j -1$ initial values determined by the tree $\K$ that we now construct.

\begin{figure}[htpb]
\begin{center}
\includegraphics[scale=0.35]{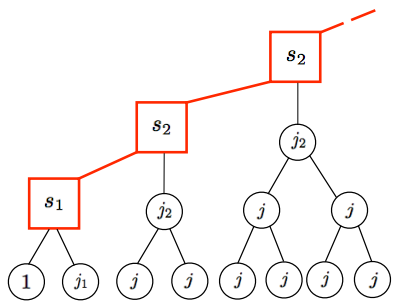}
\caption{The binary tree with label counts satisfying (\ref{2-ary}) and initial conditions beginning with $s_1+ 1$ 1s.
The entries in each node indicate the number of labels. We require $j_1 = 2j-1, j_2=j$, and $s_2 = s$.} \label{fig7}
\end{center}
\end{figure}

Here $\K$ is the infinite binary tree in \cite{JR} and $\T$ is $\K_2$ (see Figure \ref{fig7}). We traverse $\K$ in the usual way. We label $\K$ as follows: insert $s_1$ labels in the first supernode, and $s_2$ labels in all the other supernodes. Insert one label in the left child of the first supernode and $j_1$ labels in the right child of this supernode. The unique child of every other supernode contains $j_2$ labels. All other nodes in the tree get $j$ labels (see Figure \ref{fig7}).

Now we determine values for the parameters $j_1, j_2$ and $s_2$ so that the leaf label counting function for $\K$ satisfies (\ref{2-ary}). By pruning the subtree $\K(n)$ of $\K$ we mean deleting all the leaf labels of $\K(n)$ along with the nodes containing them, replacing the first supernode with a regular node containing 1 label, replacing the $s_2$ labels in second supernode with $s_1$ labels, and replacing the $j_2$ labels inside the child of the second supernode with $j_1$ labels. The resulting pruned tree is isomorphic to $\K(n-R(n) + 1-s_2 + j_1 - j_2)$, and so it contains $\pr R(n) = R(n-(s_2-1+j_2-j_1)-R(n))$ leaf labels.

Once again we have the key relation $R(n) = \pr R(n-j) + \pr R(n-2j) + \nu$ where $\nu = \alpha - 2(\beta -s_1+1-j_2+j_1) = 2j-j_1-1$. The term $\nu$ is
the difference between the number of leaf labels in $\K_1$ and $\K_2$ and twice the number of leaf labels in them after they have been pruned. Since the non-homogeneous term in (\ref{2-ary}) is 0 we must have $j_1 = 2j-1$ for $\nu$ to be 0. Furthermore, in order for $\pr R(n-j) = R(n-s-R(n-j))$ and
$\pr R(n-2j) = R(n-s-j-R(n-2j))$ we require that $s = j + s_2 - 1 + j_2 - j_1$, which simplifies to $s_2 + j_2 = s + j$. Thus, we may take $s_2 = s$ and $j_2 = j$. Then for all $n > 5j+s+s_1$, the leaf label counting function $R(n)$ satisfies (\ref{2-ary}), and $R(n)$ begins with $s_1 + 1$ 1s.\end{example}

It is worth emphasizing that the tree-based solutions we derive here for (\ref{2-ary}) are not usually the ones produced when this recursion is given \emph{exactly} $s_1+1$ 1s as the initial conditions (indeed, it is not necessarily true that any solution exists when the initial conditions are precisely $s_1+1$ 1s).\footnote{It is shown in Theorem 6.4 of \cite{IRT} that the recursion (\ref{2-ary}), together with exactly $s_1+1$ 1s as the initial conditions, where $s_1 +1 \geq s+2j$, has a well-defined solution, although no tree-based combinatorial interpretation for it could be identified.} The intuition for this is as follows: any binary tree-based solution $R(n)$ for (\ref{2-ary}) has the property that periodically it will have increments of 1 for a stretch of $2j$ indices, corresponding to that portion of the tree where we successively count the $2j$ consecutive leaf labels in the pair of leaves of the tree. But such a regularity to the increments in the solution is not usually present when the initial conditions for (\ref{2-ary}) are exactly $s_1+1$ 1s.

We now prove a necessary condition for a solution $A(n)$ of (\ref{nested}) to be the leaf label counting function for some tree $\K$ as constructed in Section \ref{sec2}. Any such $A(n)$ has the property that $A(n+1) - A(n) \in \{0,1\}$. Therefore the sequence $A$ is completely determined by its \textbf{frequency sequence} $F$ defined by $F(m) = | A^{-1}(\{m\})|$. We show that $F$ reflects the self-similarity of $\K$, in the sense that we can partition $F$ into blocks such that each block can be obtained from the previous block by a suitable transformation.

To see this, assume for simplicity that $\K$ contains one label in each regular node. Consider all values $A(n)$ as $n$ ranges over the labels in the subtree $\K_i$. Define the frequency sequence $F_i$ for that segment of $A$ by $F_i(m) = | A^{-1}(\{m\}) \cap \{n:\, n \in \K_i\}|$. We only consider the non-zero values of $F_i$ so $F_i$ is a finite sequence. For $i \geq p$, recall that $\K_{i+1}$ is obtained from $\K_i$ by adding $k$ children to each leaf of $\K_i$. It follows from this that for $i \geq p$, $F_{i+1}$ is obtained from $F_i$ as follows: first increase every value of $F_i$ by 1 except its last value (which is a 1 corresponding to the last leaf label of $\K_i$); then insert $k-1$ 1s between each successive pair of these values.

For $2 \leq i < p$, the derivation of $F_{i+1}$ from $F_i$ follows the same general procedure. However, in this range the number of children of each penultimate node in $K_{i+1}$ is not necessarily $k$, so the number of 1s inserted between pairs of values is not necessarily $k-1$. Instead it is determined by the finite tree $\T$ that is used to construct $\K$. Finally, $F_2$ and $F_1$ are determined directly from their definitions.

The partition of $F$ we seek is not given by the $F_i$ but by the sequences $F_i^*$ that are determined by removing the last value in each $F_i$ and for $i \geq 2$, increasing the first value of $F_i$ by 1. In this way they correct the frequency of $A(N(i))$. Here's why: the last value in $F_i$, which is a 1, results from the sole occurrence of $A(N(i))$ in the sequence $\{A(n);\,n \in \K_i\}$. However, from the construction of $\K$, $F(A(N(i)) = 1 + F_{i+1}(A(N(i)))= F_{i+1}^*(A(N(i)))$. The sequence $F$ is the infinite word resulting from the concatenation of all the $F_i^*$, that is, $F = \prod_{i=1}^{\infty} F_i^*$.

We illustrate the above discussion using Conolly's original recursion \[C(n) = C(n-C(n-1)) + C(n-1-C(n-2)); \;\; C(1) = 1,\; C(2) = 2\,.\] $C(n)$ counts the
number of leaf labels in the binary tree of Figure \ref{fig7} with one label per regular node and no labels in the supernodes. The frequency
sequence is $F(m) = \nu_2(2m)$ where $\nu_2(m)$ is the 2-adic valuation of $m$. We can decompose $F$ as $F_1^* = 1$, $F_2^* = 2,1$, $F_3^* = 3,1,2,1\, \ldots$ It is precisely the decomposability of the frequency sequence as above that allows one to interpret solutions to recursions of the form (\ref{nested}) as counting leaves in some infinite tree. While it is straightforward to decompose the frequency sequence of $C(n)$ (the beginning of $F_i^*$ is the first occurrence of $i$), we do not have a criterion to determine decomposability of general meta-Fibonacci sequences arising as solutions to (\ref{nested}). The problem of determining whether any tree $\T$, and hence $\K$, corresponds to a given frequency sequence appears challenging.

Our final observation is that when the initial conditions are specified by a tree $\K$ we may change the first few initial conditions arbitrarily without affecting the resulting solution sequence. Notice that if $n > N(p+1)$, pruning the tree $\K(n-ij)$ for $1 \leq i \leq j$ results in a tree containing the first $p$ subtrees $\K_1$ to $\K_p$. Suppose that for $1 \leq n \leq N(p)-1$ we set $R(n)$ arbitrarily, and for $N(p) \leq n \leq N(p+1)$ we leave $R(n)$ as the number of leaf labels in $\K(n)$. Then the recurrence relations (\ref{recursion}) or (\ref{non-homo_modif1}) will be satisfied by $R(n)$ with the new initial conditions, because the pruned trees $\pr \K(n-ij)$ for $n > N(p+1)$ and $1 \leq i \leq k$ will contain the first $N(p)$ labels. As such, all the arguments of the recursion will have value at least $N(p)$ and the proof proceeds as before. Thus, the first $N(p)-1$ values of the sequence $R(n)$ can be set arbitrarily and the recurrence relations (\ref{recursion}) or (\ref{non-homo_modif1}) still holds.

We conclude by deriving the solution for the Golomb recursion $g(n)$ \cite{Gol} discussed in Section \ref{sec1} using our tree-grafting methodology.

\begin{example}[Golomb's triangular sequence] \label{ex3} Golomb's sequence is defined by \[g(n) = g(n-g(n-1)) + 1 \;\text{and}\; g(1) = 1.\] Let $\T$ be be a rooted path of length 2. Take $s_0 =1$, $s_m=0$ for all $m \geq 1$, and $j=1$. Then we construct the unary tree whose leaf counts generate Golomb's sequence (see Figure \ref{fig8}). This shows that Golomb's sequence is a step function that increases by one at the indices $n = \binom{k+1}{2} + 1$ for every $k \geq 1$.\footnote{The step function property implies that $g(n)$ has a closed form, namely, $g(n) = \lfloor \frac{\lfloor \sqrt{8n} \rfloor + 1}{2} \rfloor$. See \cite{Gol}.}
\end{example}

\begin{figure}[htpb]
\begin{center}
\includegraphics[scale=0.35]{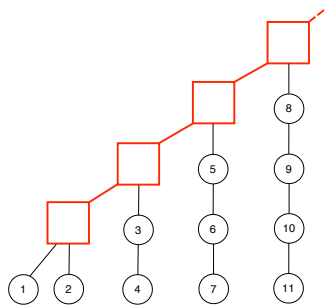}
\caption{The unary tree $\K$ that generates Golomb's recursive sequence.} \label{fig8}
\end{center}
\end{figure}

\end{document}